\documentclass[10pt]{amsart}

\usepackage{amsfonts}

\usepackage{graphicx}
\usepackage{amsmath}
\usepackage{setspace}
\usepackage[tmargin=2.0cm, bmargin=2.0cm, lmargin=3.0cm, rmargin=3.0cm]{geometry}

\usepackage{amsmath,amssymb}
\usepackage{amsthm}
\usepackage{enumerate}
\usepackage[T1]{fontenc}
\usepackage{ifxetex}
\usepackage{bbm}

\usepackage{color}
\definecolor{lw}{RGB}{255,0,0}
\definecolor{kb}{RGB}{0,255,0}

\newtheorem{thm}{Theorem}[section]
\newtheorem{prop}[thm]{Proposition}
\newtheorem{lemm}[thm]{Lemma}

\newtheorem{remark}[thm]{Remark}
\newtheorem{example}[thm]{Example}
\newtheorem{assumption}{Assumption}

\allowdisplaybreaks[4]

\def\R{{\mathbb{R}}}
\def\E{{\mathbb{E}}}
\def\P{{\mathbb{P}}}

\def\im{{\mathrm{i}}}

\def\sA{{\mathcal{A}}}

\def\sL{{\mathcal{L}}}
\def\sR{{\mathcal{R}}}
\def\sS{{\mathcal{S}}}

\def\e{{\mathrm{e}}}
\def\rd{{\mathrm{d}}}

\def\id{{\mathbf{1}}}

\DeclareMathOperator*\dist{dist}

\makeatletter
\newcommand*{\rmnum}[1]{\expandafter\@slowromancap\romannumeral #1@}
\makeatother

\title{On the exact asymptotics of exit time from a cone of an isotropic
$\alpha$-self-similar Markov process with a skew-product structure}

\author{Zbigniew Palmowski}
\address{Faculty of Pure and Applied Mathematics, Wroclaw University of Science and Technology, Wyb. Wyspia\'nskiego 27, 50-370 Wroclaw, Poland}
\email{zbigniew.palmowski@gmail.com}

\author{Longmin Wang}
\address{School of Mathematical Sciences, Nankai University, Tianjin 300071, P.R. China}
\email{wanglm@nankai.edu.cn}

\thanks{This work is partially supported by the Ministry of Science and
Higher Education of Poland under the grant 2015/17/B/ST1/01102
(2016-2019). ZP and LW kindly acknowledges partial support by the project RARE -318984, a Marie Curie IRSES Fellowship within the 7th European
Community Framework Programme.}

\date{\today}

\begin{document}

\begin{abstract}
In this paper we identify the asymptotic tail of the distribution of the exit time $\tau_C$ from a cone $C$ of an isotropic
$\alpha$-self-similar Markov process $X_t$ with a skew-product structure, that is
$X_t$ is a product of its radial process and independent time changed angular component $\Theta_t$. Under some additional
regularity assumptions, the angular process $\Theta_t$
killed on exiting from the cone $C$ has the transition density that
could be expressed in terms of
a complete set of orthogonal eigenfunctions
with corresponding eigenvalues of an appropriate generator.
Using this fact and some asymptotic properties of the exponential functional of a killed L\'evy process
related with Lamperti representation of the radial process, we prove that
$$\P_x(\tau_C>t)\sim h(x)t^{-\kappa_1}$$
as $t\rightarrow\infty$ for $h$ and $\kappa_1$ identified explicitly.
The result extends the work of DeBlassie \cite{DBla88} and Ba{\~n}uelos and Smits \cite{BS97} concerning the Brownian motion.
\vspace{3mm}


\end{abstract}

\maketitle

\pagestyle{myheadings} \markboth{\sc Z.\ Palmowski
--- L.\ Wang} {\sc Exit time of a isotropic
$\alpha$-self-similar Markov process}

\vspace{1.8cm}

\tableofcontents

\newpage

\section{Introduction}

For a dimension $d \geq 2$ and an index $\alpha > 0$  on some probability space $(\Omega, \mathcal{F}, \P_x)$ we consider
an $\R^d$-valued {\bf $\alpha$-self-similar isotropic Markov process}
$\{X_t, t\geq 0\}$,
where $\P_x(\cdot=\P(\cdot|X_0=x)$.
We recall that process $X$ is said to be $\alpha$-self-similar if for every $x \in \R^d$
and $\lambda > 0$,
\begin{center}
  \it the law of $(\lambda X_{\lambda^{-\alpha} t}, t \geq 0)$ under
  $\P_x$ is the same as $\P_{\lambda x}$.
\end{center}
Moreover, this process is said to be isotropic (or $O(d)$-invariant), if for any
$x \in \R^d$ and $\varrho \in O(d)$,
\begin{center}
  \it the law of $(\varrho(X_t), t \geq 0)$ under $\P_x$ is the same as
  $\P_{\varrho(x)}$,
\end{center}
where $O(d)$ is the group of orthogonal transformations on $\R^d$.
In this paper we assume that the radial process $R_t = |X_t|$
and the angular process $X_t/R_t$ do not jump at the same time.
Then by Liao and Wang \cite[Theorem 1]{LW11} the process $X_t$ observed up to its first hitting time of $0$
has a {\bf skew-product structure}:
\begin{equation}\label{skewstructure}X_t = R_t \Theta_{A(t)},\end{equation}
where $A(t)$ is a strictly
increasing continuous process defined by
\begin{equation}
  \label{e:At}
  A(t) = \int_0^t R_s^{- \alpha} \rd s
\end{equation}
and $\Theta_t$ is an $O(d)$-invariant Markov process on the unit
sphere $S^{d-1}$ and is {\bf independent} of the radial
process $R_t$.
The classical example concerns $d$-dimensional Brownian motion
that may be expressed as a product of a Bessel process
and a time changed spherical Brownian
motion. Moreover, the
Bessel process is independent of the spherical
Brownian motion. More generally, any continuous isotropic Markov
proces will have above representation \eqref{skewstructure} with
possibly different time change; see \cite{Gal63}.
In particular, a self-similar diffusion will have it.
Note that an isotropic self-similar Markov process might not satisfy above representation \eqref{skewstructure} though.
The most famous examples are the symmetric $(1/\alpha)$-stable L\'evy processes for
$\alpha > 1/2$. Their L\'evy measures are absolutely continuous on $\R^d\backslash\{0\}$, so their radial and angular
parts may jump together, and thus do not possess a skew product structure as defined above.

We will also consider an {\bf open cone} $C$ in $\R^d$ generated by
a domain $D$  in the unit sphere $S^{d-1}$, that is
$C = \cup_{r > 0} r D$.
We define the {\bf first
exit time} of $X_t$ from the cone $C$ by
\begin{equation}
  \label{e:exittime}
  \tau_C = \inf\{ t > 0:\ X_t \not\in C \}.
\end{equation}
The purpose of this paper is to study the {\bf asymptotic behavior of the
exit probability} $\P_x (\tau_C > t)$ as $t \to \infty$ for $x \in C$.
In fact we prove that
\begin{equation}\label{asmain}
\P_x(\tau_C>t)\sim h(x)t^{-\kappa_1}\end{equation}
as $t\rightarrow\infty$ for $h$ and $\kappa_1$ identified explicitly,
where we write $f(t)\sim g(t)$ for some positive functions $f$ and $g$ iff $\lim_{t\to\infty}f(t)/g(t)=1$.

The main idea of the proof is based on the following steps. In the first one we give the following representation
\begin{equation*}
  q_D(t,\theta,\eta) = \sum_{j = 1}^{\infty} \e^{- \lambda_j t}
  m_j(\theta) m_j(\eta),
\end{equation*}
of the transition density for the angular process $\Theta_t$
killed upon exiting from the cone $C$ in terms of
orthogonal eigenfunctions $m_j$
with corresponding eigenvalues $\lambda_j$ of
$- \sS \big|_D$ for the generator
$\sS$ of $\Theta_t$ restricted to $D$ with  Dirichlet boundary condition.
Then
\begin{equation*}
\P_x(\tau_C > t) = \sum_{j=1}^{\infty} m_j(x/|x|) \left( \int_D m_j \rd \sigma
  \right) \E_{|x|}
  \left[ \e^{- \lambda_j A(t)}, t < T_0 \right]
\end{equation*}
for $\sigma$ being the
normalized surface measure on $S^{d - 1}$, where
\begin{equation}
  \label{e:zeta}
  T_0 = \inf\{t > 0:\ X_t = 0 \} = \inf\{ t > 0:\ R_t = 0\}.
\end{equation}
Using Lamperti \cite{Lam72} transformation we can express process $\{R_t, t<T_0\}$
as a time change of the exponential of an $\R \cup
\{-\infty\}$-valued L\'evy process, that is,
there exists an $\R \cup \{- \infty\}$-valued
L\'evy process $\xi_t$ starting from $0$ and with lifetime $\zeta$,
whose law does not depend on $|x|$, such that
\begin{equation}
  \label{e:pssMp}
  R_t = |x| \exp \left( \xi_{A(t)} \right), \quad 0 \leq t < T_0.
\end{equation}
This gives the following representation of the tail exit probability:
\begin{equation*}
  \label{e:exitprobs}
  \P_x(\tau_C > t) = \sum_{j = 1}^{\infty} m_j(x / |x|) \left( \int_D
    m_j \rd \sigma \right) \P \left( I_{e_{\lambda_j}}(\alpha \xi) >
    |x|^{- \alpha} t \right),
\end{equation*}
where
\begin{equation}
  \label{e:Tt}
  I_t(\alpha \xi) = \int_0^t \exp(\alpha \xi_s) \rd s
\end{equation}
and
  $I(\alpha \xi) = \lim_{t \to \zeta} I_t(\alpha \xi)$ is an exponential functional. The final result \eqref{asmain} follows from
Rivero \cite[Lemma 4]{Riv05a} and Maulik and Zwart \cite[Theorem 3.1]{MZ06}
concerning tha tail asymptotics of the exponential functional
$I(\alpha \xi)$.
In this case $\kappa_1$ solves equation
\begin{equation}\label{kappajeden}
\phi(\alpha \kappa_1) = \lambda_1\end{equation}
for the Laplace exponent of the process $\xi$.

Our main result \eqref{asmain} extends the work of
DeBlassie \cite{DBla88} and Ba{\~n}uelos and Smits \cite{BS97} concerning the Brownian motion
(see also \cite{Hernandez} for $\alpha$-stable process case).

The asymptotics \eqref{asmain} determines also
the critical exponents of integrability of the exit time $\tau_C$.
In this sense it generalizes series of papers concerning $\alpha$-stable process, see
Kulczycki \cite{kulcz} and Ba{\~n}uelos and Bogdan \cite{BB04} and references therein.

The paper is organized as follows.
In Preliminaries we give and prove main facts used later.
In the next Section \ref{sec:main} we give the main result and its proof.

\section{Preliminaries}\label{sec:prel}

\subsection{Skew-product structure}
\label{s:levy}

Let $X_t$ be an $\alpha$-self-similar isotropic Markov process
with skew-product representation \eqref{skewstructure}.
The process $\Theta_t$ is an $O(d)$-invariant Markov process on $S^{d - 1}$ with
transition semigroup $Q_t$ and infinitesimal generator $\sS$.

Throughout this paper, we assume that
\begin{assumption}
  \label{a1}
  $\Theta_t$ possesses a bounded
transition density $q(t,\theta,\eta)$ with respect to $\sigma$, the
normalized surface measure on $S^{d - 1}$, and there exist positive
constants $C$ and $\beta$ such that
\begin{equation}
  \label{e:hkbound}
  q(t,\theta,\eta) \leq C t^{- \beta}
\end{equation}
for all $(t, \theta,\eta) \in (0, \infty) \times S^{d - 1} \times S^{d
- 1}$.
\end{assumption}

\begin{example}[Brownian motion]\rm
In the case when $\Theta_t$ is a Brownian motion
the Assumption \ref{a1} is satisfied. Indeed,
the generator $\sS$ of $\Theta_t$ is a multiple of the Laplace-Beltrami operator
$\Delta_{S^{d - 1}}$ on $S^{d -
  1}$. Moreover, it is known that the transition density
$h(t,\theta,\eta)$ of $\Theta_t$
 has the Gaussian upper bound:
\begin{equation}\label{gaussianbound}
 h(t,\theta,\eta) \leq c_1 t^{- \frac{d - 1}{2}} \e^{- c_2
   \frac{d(\theta,\eta)^2}{t}}, \quad t > 0, \ \theta, \eta \in S^{d-1}
\end{equation}
for some positive constants $c_1$ and $c_2$; .
\end{example}

\begin{example}[Subordinate Brownian motion]\rm
  \label{E:sbm}
  Fix $\gamma \in (0, 1)$. Let $W_t$ be a Brownian motion on $S^{d -
    1}$ with transition density $h(t,\theta,\eta)$. Let $S_t$ be a
  $\gamma$-stable subordinator, i.e., a
  L\'evy process in $\R$, supported by $[0, \infty)$, with Laplace
  transform
  \[ \E \left[ \e^{- \vartheta S_t} \right] = \exp (- t
  \vartheta^{\gamma}), \quad \vartheta > 0. \]
  The Assumption \ref{a1} is also satisfied when $\Theta_t =
  W_{S_t}$.
  Indeed, in this case $\Theta_t$
  is an $O(d)$-invariant pure
  jump Markov process on $S^{d - 1}$ with transition density
  \[ q(t,\theta, \eta) = \int_0^{\infty} h(u,\theta, \eta) p_t(u) \rd
    u, \]
  where $p_t(u)$ is the probability
  density of $S_t$.
  By Theorem 37.1 of Doetsch \cite{Doe74},
  \begin{equation}\label{limitu}
   \lim_{u \to \infty} p_1(u) u^{1 + \gamma} =
  \frac{\gamma}{\Gamma(1 - \gamma)}. \end{equation}
  The limit \eqref{limitu} together with the  scaling
  property:
  \[ p_t(u) = t^{- \frac{1}{\gamma}} p_1(t^{- \frac{1}{\gamma}} u) \]
  give the following upper bound:
  \[ p_t(u) \leq c_1 t u^{- 1 - \gamma}, \quad t, u > 0. \]
  Using \eqref{gaussianbound} one can observe now that
  \[ q(t,\theta,\eta) \leq c_3 t^{- \frac{d - 1}{2 \gamma}} \wedge
  \frac{t}{d(\theta,\eta)^{d - 1 + 2 \gamma}}\]
  for $(t,\theta,\eta)
  \in (0, \infty) \times S^{d - 1} \times S^{d - 1}$.
\end{example}

We will now give sufficient conditions for the Assumption \ref{a1} to be
satisfied for general $O(d)$-invariant Markov
process $\Theta_t$ on $S^{d - 1}$.
If we identity $S^{d - 1}$ with $O(d) / O(d - 1)$ then
$\Theta_t$ may be viewed as a L\'evy process on the compact homogeneous
space $O(d) / O(d - 1)$. Furthermore,  the generator $\sS$
of $\Theta_t$ was given by Hunt \cite{Hun56}
(see also Liao \cite{Lia04a}) explicitly.
We state this result as follows.
Let $C^{\infty}(S^{d - 1})$ be the space of smooth functions on $S^{d
  - 1}$
and let  $\pi: O(d)\rightarrow S^{d - 1}$  be the
map $g\rightarrow g \overline{o}$ for $\overline{o}=(0,\ldots, 0,1)\in \R^d$.
Restricted to a sufficient small neighborhood $V$ of $\overline{o}$, the map
\[\varphi : (y_1, \ldots, y_d ) \rightarrow \pi(e^{
\sum_{j=1}^d y_j O_j})\]
is a diffeomorphism and $(y_1, \dots , y_d)$ may be used as local coordinates on $\varphi(V)$, where
$(O_1,\dots, O_d)$ is a basis of Lie algebra of $O(d)$.
Then by \cite[Theorem 2.2]{Lia04a} the domain of $\sS$
contains $C^{\infty}(S^{d - 1})$ and for $f \in C^{\infty}(S^{d -
  1})$,
\begin{equation}
  \label{e:S0}
  \sS f(o) = T f(o) + \int_{S^{d - 1}}
  \left(f(\theta) - f(o)-\sum_{j=1}^d y_j(\theta)\frac{\partial f(o)}{\partial y_j} \right)\nu(\rd \theta),
\end{equation}
where $o$ is the origin on $S^{d - 1}$, $T$ is an $O(d)$-invariant
second order differential operator on $S^{d-1}$ and $\nu$ is an $O(d -
1)$-invariant measure on $S^{d - 1}$, called
the L\'evy measure of $\Theta_t$, that satisfies $\nu(\{o\}) = 0$ and
\begin{equation}
  \label{e:levymeas}
  \int_{S^{d - 1}} [\dist(\theta, o)]^2 \nu(\rd \theta) < \infty.
\end{equation}
Since $O(d) / O(d-1)$ is irreducible, all the $O(d)$-invariant second
order differential operators are multiples of
$\Delta_{S^{d-1}}$. Therefore we may rewrite \eqref{e:S0} as
\begin{equation}
  \label{e:S}
  \sS f(o) = a \Delta_{S^{d-1}} f(o) + \int_{S^{d - 1}}
  \left(f(\theta) - f(o)-\sum_{j=1}^d y_j(\theta)\frac{\partial f(o)}{\partial y_j} \right)\nu(\rd \theta)
\end{equation}
for some $a \geq 0$.

Note that when $\Theta_t$ is a subordinate Brownian motion defined in
Example \ref{E:sbm}, we have $a = 0$ and the L\'evy measure
\[ \nu \asymp d(\theta,o)^{- d + 1 - 2 \gamma} \quad \text{near }
o. \]
As observed in this case the Assumption \ref{a1} holds true.

This phenomenon holds for more general $\Theta_t$.
Using \cite[Theorems 3 and 6]{LW07d} we can state the following proposition
giving sufficient conditions for the Assumption \ref{a1} to hold true.


\begin{prop}
$\Theta_t$ is a L\'evy process on $S^{d - 1}$ with the infinitesimal
  generator $\sS$ given by \eqref{e:S}. Assume that either $a > 0$ or the
  L\'evy measure $\nu$ is asymptotically larger than $d(\theta,o)^{-
    \gamma}$ near $\theta = o$ for some $\gamma \in (d - 1, d - 1 +
  2)$. Then $\Theta_t$
  has a bounded transition density $q(t,\theta,\eta)$ and it satisfies
  Assumption \ref{a1}.
\end{prop}


For any open subset $D \subset S^{d - 1}$ we define the first exit time
of $\Theta_t$ from $D$ by
\begin{equation}
  \label{e:exittimeD}
  \tau^{\Theta}_D = \inf\{t > 0:\ \Theta_t \not\in D\}.
\end{equation}
Let $\Theta^D$ be the killed process of $\Theta$ upon exiting from $D$, that is,
$\Theta^D_t = \Theta_t$ if $t < \tau^{\Theta}_D$ and $\Theta^D_t
= \partial$ if $t \geq \tau^{\Theta}_D$, where $\partial$ is a
cemetery state. Its infinitesimal generator is $\sS \big|_D$, the
restriction of $\sS$ to $D$ with the Dirichlet boundary condition.
Then
\begin{equation}
  q_D(t,\theta,\eta) = q(t,\theta,\eta) - \E_{\theta} \left[ q(t -
    \tau^{\Theta}_D, \Theta_{\tau^{\Theta}_D},\eta); \tau^{\Theta}_D <
  t\right]
\end{equation}
is the transition density of $\Theta^D$.
Clearly, $q_D(t,\theta,\eta)
\leq q(t,\theta,\eta)$ for all $t > 0$ and $\theta$, $\eta \in S^{d -
  1}$. As a consequence, the transition semigroup $Q^D_t$
associated to the subprocess $\Theta^D$ is
 compact on $L^2$. Let
$\{\lambda_j\}_{j=1}^{\infty}$ be the
eigenvalues of $- \sS \big|_D$ written in increasing order and
repeated according to its multiplicity, and $m_j$ the corresponding
eigenfunctions normalized by $\Vert m_j \Vert_2 = 1$. Then by
\cite[Theorem 2.1.4]{Dav90}, $m_j \in L^{\infty}$ for all $j$ and
$Q^D_t$ has a transition density $q_D(t,\theta,\eta)$, which can be
represented as the series:
\begin{equation}
  \label{e:td}
  q_D(t,\theta,\eta) = \sum_{j = 1}^{\infty} \e^{- \lambda_j t}
  m_j(\theta) m_j(\eta)
\end{equation}
that converges uniformly on $[\delta, \infty) \times D
\times D$ for all $\delta > 0$.

\begin{remark}\rm
  Note that we do not assume any regularity condition on the boundary
  $\partial D$ of $D$. Thus $q_D(t,\theta,\eta)$ (or $m_j(\theta)$) need
  not vanish continuously on the boundary $\partial D$.
\end{remark}

\begin{remark}\rm
  If $S^{d - 1}\setminus \overline{D}$ is not empty, then from
  the monotonicity of Dirichlet eigenvalues we have that
$\lambda_1 > 0$; see \cite[Section I.5]{Cha84} for more details (check also Lemma \ref{lowerboundlemma} below).
\end{remark}

\begin{remark}\rm
  Assume that $q_D(t,\theta,\eta)$ is strictly positive for $t > 0$ and
  $\theta$, $\eta \in D$. Then we have from
  Jentzsch's theorem (\cite[Theorem V.6.6]{Sch74})  that $\lambda_1$ is
  a simple eigenvalue for $- \sS \big|_D$. Using the standard arguments, like the ones
  given in the proof of \cite[Theorem 2.4]{CZ95a}, one can show
  $q_D(t,\theta,\eta)$ is strictly positive when $\Theta_t$ is a
  Brownian motion on $S^{d - 1}$ and $D$ is connected or $\Theta_t$
  is a subordinate Brownian motion on $S^{d - 1}$ satisfying the
  conditions of Example \ref{E:sbm}.
\end{remark}

In our analysis the crucial fact is the following
lower bound for the eigenvalues $\lambda_j$ ($j\geq 1$).
\begin{lemm}\label{lowerboundlemma}
  Assume \eqref{e:hkbound} holds true. Then for every $j \geq 1$, we have
  \begin{equation}
    \label{e:lbev}
    \lambda_j \geq \left[ C \sigma(D) \right]^{- \frac{1}{\beta}}
    j^{\frac{1}{\beta}}
  \end{equation}
  and
  \begin{equation}
    \label{e:ubef}
    \Vert m_j \Vert_{\infty} \leq \e C \left[  \sigma(D) \right]^{\frac{1}{2}}
    \lambda_j^{\beta}.
  \end{equation}
\end{lemm}

\begin{proof}
We will follow the same idea as the one that the proof of \cite[Lemma 2.7]{DL82} is based on.
In particular, since $\lambda_j$ is ordered increasingly, we have from \eqref{e:td}
  and \eqref{e:hkbound} that
  \[ j \e^{- \lambda_j t} \leq \sum_{j = 1}^{\infty} \e^{- \lambda_j t} = \int_D
  q_D(t,\theta,\theta) \sigma(\rd \theta) \leq C \sigma(D) t^{- \beta}. \]
  Taking $t = \lambda_j^{- 1}$ we obtain $j \leq C \sigma(D) \lambda_j^{\beta}$
  and \eqref{e:lbev} follows immediately.

  Note that
    \[ m_j(\theta) = \e^{\lambda_j t} Q^D_t m_j(\theta) = \e^{\lambda_j
    t} \int_D q_D(t,\theta,\eta) m_j(\eta) \sigma(\rd \eta). \]
  By the Cauchy-Schwarz inequality,
    \[ \Vert m_j \Vert_{\infty} \leq \e^{\lambda_j t} \sup_{\theta}
    \left( \int_D q_D(t,\theta,\eta)^2 \sigma(\rd
    \eta) \right)^{1 / 2} \left( \int_D m_j(\eta)^2 \sigma(\rd \eta)
         \right)^{1/2}
    \leq C  \left[  \sigma(D) \right]^{\frac{1}{2}} \e^{\lambda_j t}
    t^{- \beta}. \]
  The proof is completed by setting $t = \lambda_j^{- 1}$.
\end{proof}

\subsection{Positive self-similar Markov processes}

Recall that $R_t = |X_t|$ is a positive ($\R_+$-valued)
$\alpha$-self-similar Markov process starting at $|x|$. According
to Lamperti \cite{Lam72}, up to its first hitting time of $0$, $R_t$
may be expressed as a time change of the exponential of an $\R \cup
\{-\infty\}$-valued L\'evy process.
More formally,  there exists an $\R \cup \{- \infty\}$-valued
L\'evy process $\xi_t$ starting from $0$ and with lifetime $\zeta$,
whose law does not depend on $|x|$, such that
\begin{equation}
  \label{e:pssMp}
  R_t = |x| \exp \left( \xi_{A(t)} \right), \quad 0 \leq t < T_0,
\end{equation}
where $T_0$ is the first hitting time of $0$ by $R$ defined formally
in
\eqref{e:zeta} and $A(t)$ is the positive continuous functional given
by $A(t) = \int_0^t R_s^{\alpha} \rd s$.
The law of $\xi$ is characterized completely by its L\'evy-Khintchine
exponent
\begin{equation}
  \label{e:Psi}
  \Psi(z) = \log \E \left[ \e^{\im z \xi_1} \right] = - q + \im b z -
  \frac{\sigma^2}{2} z^2 + \int_{-\infty}^{+\infty} \left(\e^{\im z y}
    - 1 - \im z y \id_{\{|y| < 1\}} \right) \Pi(\rd y),
\end{equation}
where $q \geq 0$, $\sigma \geq 0$, $b \in \R$ and $\Pi$ is a L\'evy
measure satisfying the condition $\int_{\R} \left( 1 \wedge |y|^2
\right) \Pi(\rd y) < \infty$. The lifetime $\zeta$ of $\xi$ is an
exponential random variable with parameter $q$, with the convention
that $\zeta = \infty$ when $q = 0$. Observe that the
process $\xi$ does not depend on the starting point of $X$. Hence we
will denote the law of $\xi$ by $\P$.

For fixed $\alpha > 0$, we define the exponential functional
$I_t(\alpha \xi)$ by \eqref{e:Tt}.
Then by a change of variable $s = A(u)$,
\[ I_t(\alpha \xi) = \int_0^{A^{-1}(t)} \exp(\alpha \xi_{A(u)}) R_u^{-\alpha} \rd
u = |x|^{-\alpha} A^{-1}(t). \]
Hence $A(|x|^{\alpha} t)$ is the right inverse of the strictly increasing
continuous process $I_t(\alpha \xi)$ and we can recover the law of $(R_t, t < T_0)$ from
the law of $\xi_t$ for fixed $|x|$ and $\alpha > 0$.
In particular, we have
\begin{equation}
  (T_0, \P_x) \stackrel{d}{=} (|x|^{\alpha} I(\alpha \xi), \P).
\end{equation}

As mentioned in \cite{Lam72}, the probabilities $\P_x(T_0 = +
\infty)$, $\P_x(T_0 < + \infty, R_{T_0-} = 0)$ and $\P_x(T_0 < +
\infty, R_{T_0 - } > 0)$ are $0$ or $1$ independently of $x$.
Moreover, we have
\begin{enumerate}[(i)]
\item if $\P_x(T_0 = + \infty) = 1$, then $\zeta = + \infty$,
  $\limsup_{t\to \infty} \xi_t = + \infty$, and $\lim_{t\to \infty}
  A(t) = + \infty$;
\item if $\P_x(T_0 < + \infty, R_{T_0 - } = 0) = 1$, then $\zeta = +
  \infty$, $\lim_{t\to\infty} \xi_t = - \infty$, and $\lim_{t \to
    T_0-} A(t) = + \infty$;
\item if $\P_x(T_0 < + \infty, R_{T_0 - } > 0) = 1$, then $\zeta$ is
  an exponentially distributed random time with parameter $q > 0$.
  Moreover, $A(T_0-)$ has the same distribution as that of $\zeta$,
  thus the functional $A(t)$ always jumps from a finite value to
  $+\infty$, that is, $\P_x(A(T_0-) < + \infty, A(T_0) = + \infty) =
  1$.
\end{enumerate}


Let $e_{\lambda}$ be an independent exponential random variable with
parameter $\lambda$. Then
\[ \E_{|x|} \left[ \e^{-\lambda A(t)}, t < T_0 \right] = \P_{|x|}
\left( A(t) < e_{\lambda}, t < T_0 \right). \]
Note that by the construction above we have that $t < T_0$ is
equivalent to $A(t) < \zeta$. Thus
\begin{equation}\label{secondstep}
  \E_{|x|} \left[ \e^{-\lambda A(t)}, t < T_0 \right] = \P_{|x|}
  \left( A(t) < e_{\lambda} \wedge \zeta \right) = \P \left(
    \int_0^{e_{\lambda}} \exp (\alpha \xi_s) \rd s >
    |x|^{-\alpha} t  \right),
\end{equation}
where in the last equality we used the fact that $|x|^{\alpha}
I_t(\alpha \xi)$ is the right inverse of $A(t)$.
The equation \eqref{secondstep} will give
(apart of the representation \eqref{e:td}) another main ingredient of the proof of the main result.
In the last step we will need the tail asymptotic behaviour of
the exponential function $I_{e_\lambda}(\alpha \xi)$ described below.

\subsection{Exponential functional of a killed L\'evy process}

Let $\xi^\lambda$ be
a L\'evy process with L\'evy-Khintchine exponent
$\Psi$ given by \eqref{e:Psi} killed by the
independent exponential time $e_{\lambda}$ with parameter
$\lambda > 0$. Thus the resulted process has a lifetime $\zeta' =
e_{\lambda} \wedge \zeta$, an exponential random variable with
parameter $\lambda + q$.

We define the Laplace exponent of $\xi^\lambda$ via
\begin{equation}
  \label{e:lt}
  \E \left( \exp( \vartheta \xi_t^\lambda) \right) = \exp(t \phi(\vartheta)),
  \quad t \geq   0,\quad \vartheta \in \Xi,
\end{equation}
where $\Xi=\{\vartheta: \phi(\vartheta)<\infty\}$.
By \eqref{e:Psi}, for $\vartheta\in\Xi$ we have
\begin{equation}
  \label{e:phi}
  \phi(\vartheta) = \Psi(- \im \vartheta) = - (q+\lambda) + b \vartheta
  + \frac{\sigma^2}{2} \vartheta^2 + \int_{-
    \infty}^{\infty} \left( \e^{\vartheta y} - 1 - \vartheta y \id_{\{| y | < 1\}}
  \right) \Pi(\rd y)
\end{equation}
for some parameters $b,\sigma$ and a L\'evy measure $\Pi$.
It is easy to see from H\"older inequality that $\phi(\vartheta)$ is a convex
function. From now we assume that $\xi$
satisfies the following conditions.

\begin{assumption}
  \label{a2}
  $\xi^\lambda$ is not arithmetic.
\end{assumption}

\begin{assumption}
  \label{a3}
  There exists a constant $
  \vartheta^{*} >  \alpha \left( 1 \vee (2\beta) \right)$ such that $\phi(\vartheta) < \infty$ for
  $0 < \vartheta < \vartheta^{*}$ and $\lim_{\vartheta \to
    \vartheta^{*}} \phi(\vartheta) = \infty$, where $\beta$ is the
  constant in Assumption \ref{a1}.
\end{assumption}

Note that under Assumption \ref{a3}, we have that
for every $\lambda > 0$, there exists a unique $0 < \kappa
<\vartheta^{*}$ such that
\begin{equation}
  \phi(\alpha \kappa) = \lambda.
\end{equation}
Moreover,
\[ \E[\xi_1^\lambda \e^{\alpha \kappa \xi_1^\lambda}] =
\phi'(\alpha \kappa)<+\infty.\]

In the proof we will use the following crucial result
giving the tail asymptotics of the distribution of the exponential functional.

\begin{thm}[{\cite[Lemma 4]{Riv05a}}, {\cite[Theorem 3.1]{MZ06}}]
  \label{T:expmoment}
  Suppose that the Asuumptions \ref{a2} and \ref{a3} are satisfied.
  Then, as $t \to \infty$,
  \begin{equation}
    t^{\kappa} \P(I_{e_\lambda}(\alpha \xi) > t) \sim   \frac{1}{\alpha \phi'(\alpha \kappa)} \E \left[ I_{e_{\lambda}}(\alpha
  \xi)^{\kappa - 1} \right].
  \end{equation}
\end{thm}

\begin{example}[Linear Brownian motion with drift]\rm
  \label{ex:dbm}
  Let $\sigma > 0$, $b \in \R$ and $\xi_t = \sigma B_t + b t$, where
  $B_t$ is a standard linear Brownian motion. Then $\phi(\vartheta) =
  \frac{\sigma^2 \vartheta^2}{2} + b \vartheta$ and for
  \[ \kappa = \frac{1}{\alpha \sigma^2} \left( \sqrt{2 \sigma^2
      \lambda + b^2} - b   \right), \]
  we have $\E \left[ \e^{\alpha \kappa \sigma (B_1 + b)}; 1 < e_{\lambda}
  \right] = 1$. By Theorem \ref{T:expmoment},
  \[ \lim_{t \to \infty} t^{\kappa} \P(I_{e_{\lambda}}(\alpha \xi) > t) =
  C_{\kappa}, \]
  where
  \[ C_{\kappa} = \frac{1}{\alpha \sqrt{2 \sigma^2 \lambda + b^2}} \E
  \left\{ \left[ \int_0^{e_{\lambda}} \exp \left( \alpha \sigma B_s +
        \alpha b s \right) \rd s \right]^{\kappa - 1} \right\}. \]
  By the scaling property of Brownian motion, the random variable
  \[ \int_0^{e_{\lambda}} \exp \left(\alpha \sigma B_s + \alpha b s \right)
  \rd s \]
  has the same distribution as the integral:
  \[ \int_0^{\frac{\alpha^2 \sigma^2}{4} e_{\lambda}} \exp \left( 2 B_s +
    \frac{4 b}{\alpha \sigma^2} s \right) \rd s. \]
  Note that $\frac{\alpha^2 \sigma^2}{4} e_{\lambda}$ is an exponential
  distributed random variable with parameter $\frac{4}{\alpha^2 \sigma^2}
  \lambda$ independent of $B_t$. Yor \cite{Yor92} (see also
  \cite[Theorem 4.12]{MY05b}) proved the following the identity in law:
  \begin{equation}
    \label{e:expfun}
    I_{e_{\lambda}}(\alpha \xi) \stackrel{\mathrm{d}}{=}
    \frac{Z_{1,a}}{2 \gamma_{\kappa}},
  \end{equation}
  where $a = \kappa + \frac{2 b}{\alpha \sigma^2}$, $Z_{1,a}$ is a beta
  variable with parameters $(1,a)$, and $\gamma_{\kappa}$ is a gamma
  variable with parameter $\kappa$, which is independent of
  $Z_{1,a}$. Since
  \[ \E \left[ Z_{1,a}^{\kappa - 1} \right] = \int_0^1 t^{\kappa - 1}
  a (1 - t)^{a - 1} \rd t = \frac{\Gamma(\kappa) \Gamma(a + 1)}{\Gamma(a
    + \kappa)} \]
  and
  \[ \E \left[ \gamma_{\kappa}^{1 - \kappa} \right] =
  \frac{1}{\Gamma(\kappa)} \int_0^{\infty} t^{1 - \kappa} t^{\kappa -
    1} \e^{- t} \rd t = \frac{1}{\Gamma(\kappa)}, \]
  we have
  \[ \E \left[ \left( I_{e_{\lambda}}(\alpha \xi) \right)^{\kappa - 1}
  \right] = \frac{2^{1 - \kappa} \Gamma(a + 1)}{\Gamma(\kappa + a)} =
  \frac{2^{1 - \kappa} \Gamma \left( \kappa + \frac{2 b}{\alpha \sigma^2} + 1
    \right)}{\Gamma \left( 2 \kappa + \frac{2 b}{\alpha \sigma^2}
    \right)}. \]
  Therefore,
  \[ C_{\kappa} = \frac{4}{\alpha^2 \sigma^2 2^{\kappa}} \,\, \frac{\Gamma
    \left( \kappa + \frac{2 b}{\alpha \sigma^2} +
      1 \right)}{\Gamma \left(2 \kappa +
      \frac{2 b}{\alpha \sigma^2} + 1 \right)}. \]
\end{example}

\section{Main result}\label{sec:main}

Let
\begin{equation}
  \label{e:Mx}
  M(x) = \sum_{j: \lambda_j = \lambda_1} \left( \int_D m_j \rd \sigma
  \right) m_j(x / |x|)
\end{equation}
be a particular eigenfunction corresponding to the eigenvalue $\lambda_1$ of the
operator $\sS$ in $D$ with  Dirichlet boundary condition.
Moreover, let $\kappa_1$ solves
$\phi(\alpha \kappa_1) = \lambda_1$, that is, \eqref{kappajeden} is satisfied.

Recall that $\tau_C$ is the exit time for the cone $C$ of the $\alpha$-self-similar Markov process $X_t$ with a skew-product structure
\eqref{skewstructure}.
The main result of this paper is the following asymototics.
\begin{thm}
  \label{T:asymp}
  Under the Assumptions \ref{a1}, \ref{a2} and \ref{a3}, we have,
  \begin{equation}
    \label{e:asymp}
    \P_x(\tau_C > t) \sim \frac{1}{\alpha \phi'(\alpha \kappa_1)} \E \left[
      I_{e_{\lambda_1}}(\alpha \xi)^{\kappa_1 - 1} \right] M(x) \left(
      |x|^{-\alpha} t \right)^{- \kappa_1},
  \end{equation}
  as $t \to \infty$.
\end{thm}

\begin{remark}\rm
  $M(x)$ does not depend on the choices of eigenfunctions $m_j$ with
  $\lambda_j = \lambda_1$. Indeed, if we have another choice $m'_j$,
  then there exists an orthogonal matrix $(a_{ij})$ such
 that $m'_i = \sum_j a_{ij} m_j$, which is equivalent to $m_j = \sum_i
 a_{ij} m'_i$. Thus,
 \begin{align*}
   \sum_i \left( \int_D m'_i \rd \sigma \right) m'_i
   = & \sum_i \left(
       \sum_j \int_D a_{ij} m_j \rd \sigma \right) m'_i \\
   = & \sum_j \left(
       \int_D m_j \rd \sigma \right) \sum_i a_{ij} m'_i \\
   = & \sum_j \left( \int_D m_j \rd \sigma \right) m_j.
 \end{align*}
\end{remark}

\begin{example}\rm
  \label{ex:bmcone}
  Assume $X_t$ is an isotropic $\alpha$-self-similar diffusion process
  on $\R^d$. Then the radial process $R_t = |X_t|$ is a positive
  $\alpha$-self-similar diffusion process and $\Theta_t$ is a (possibly
  nonstandard) Brownian motion on $S^{d - 1}$ with $a \Delta_{S^{d -
      1}}$ as its infinitesimal generator for some $a > 0$.
      Using the Lamperti's relation, we have $\xi_t = \sigma B_t + b t$ for some
  $\sigma > 0$ and $b \in \R$,  where $B_t$ is a standard
  Brownian motion. Clearly, all the Assumptions \ref{a1}, \ref{a2} and
  \ref{a3} are satisfied. It follows from Example \ref{ex:dbm} that
  \[ \lim_{t \to \infty} t^{- \kappa_1} \P_x \left( \tau_C > t
      \right) = \frac{4}{\alpha^2 \sigma^2} \frac{\Gamma \left(
          \kappa_1 + \frac{2 b}{\alpha \sigma^2} + 1 \right)}{\Gamma
        \left( 2 \kappa_1 + \frac{2 b}{\alpha \sigma^2} + 1 \right)}
      \left( \frac{|x|^2}{2} \right)^{\kappa_1} M(x), \]
  where $M(x)$ is defined by \eqref{e:Mx}. In particular, when $X_t$
  is a $d$-dimensional Brownian motion, we have $\alpha = 2$, $\sigma
  = 1$, $b = \frac{d}{2} - 1$, and $a = \frac{1}{2}$. Thus
  \begin{equation}
    \label{e:asympbm}
      \lim_{t \to \infty} t^{- \kappa_1} \P_x \left( \tau_C > t
      \right) =  \frac{\Gamma \left( \kappa_1 + \frac{d}{2}
        \right)}{\Gamma \left( 2
      \kappa_1 + \frac{d}{2} \right)} \left( \frac{|x|^2}{2}
  \right)^{\kappa_1} M(x),
  \end{equation}
where
\[ \kappa_1 = \frac{1}{2} \left( \sqrt{2 \lambda_1 + \left(
        \frac{d}{2} - 1 \right)^2} - \left( \frac{d}{2} - 1 \right)
  \right). \]
This recovers the seminal result of De Blassie \cite{DBla88} (see also \cite[Corollary
1]{BS97}).
\end{example}

\begin{proof}[Proof of Theorem \ref{T:asymp}]
We start the proof from the observation that $\tau_C$ is just the first time $t$ that $R_t = 0$ or
the  angular process $\Theta_{A(t)} \not\in D$, that is:
\begin{equation}
  \P_x (\tau_C > t) = \P_x (T_0 > t,\ \tau^{\Theta}_D > A(t)).
\end{equation}
By the assumed independence of $R_t$ and $\Theta_t$ in \eqref{skewstructure} we have,
\begin{equation}
  \label{e:exitprob}
  \P_x(\tau_C > t) = \int_0^{\infty} \P_{x/|x|}(\tau^{\Theta}_D > u)
  \rd_u \P_{|x|}(A(t) \leq u, t < T_0).
\end{equation}
By \eqref{e:td} the exit probability $\P_{\theta}(\tau^{\Theta}_D >
t)$ can be represented as
\begin{equation}
  \P_{\theta}(\tau^{\Theta}_D > t) = \int_D q_D(t,\theta,\eta) \rd
  \sigma (\eta) = \sum_{j=1}^{\infty} \e^{-\lambda_j t} m_j(\theta)
  \int_D m_j \rd \sigma.
\end{equation}
Thus by \eqref{secondstep}:
\begin{equation}
  \begin{aligned}
  \P_x(\tau_C > t) = & \sum_{j=1}^{\infty} m_j(x/|x|) \int_D m_j \rd
  \sigma \int_0^{\infty} \e^{-\lambda_j u} \rd_u \P_{|x|}(A(t) \leq u,
  t < T_0)
\\
  = & \sum_{j=1}^{\infty} m_j(x/|x|) \left( \int_D m_j \rd \sigma
  \right) \E_{|x|}
  \left[ \e^{- \lambda_j A(t)}, t < T_0 \right]\\
  =&\sum_{j = 1}^{\infty} m_j(x / |x|) \left( \int_D
    m_j \rd \sigma \right) \P \left( I_{e_{\lambda_j}}(\alpha \xi) >
    |x|^{- \alpha} t \right).\label{e:exitprobs}
\end{aligned}
\end{equation}
For $j \geq 1$ let $\kappa_j$ ($j\geq 1$) be the solutions of
\begin{equation}\label{kappaj}
\phi(\alpha \kappa_j) =
\lambda_j.
\end{equation}
Since $\lambda_j \to \infty$, we have that $\liminf_{j \to \infty}
\kappa_j \geq \vartheta^{*}/\alpha$. Fix a $\kappa_0$ with the property
\begin{equation}\label{kappazero}
1 \vee \kappa_1 \vee (2\beta) < \kappa_0 <
\vartheta^{*}/\alpha.
\end{equation}
Then there are only a finite number of $j$'s (for $j\geq 1$), being in the set, say, $A$, such
that $\kappa_1<\kappa_j \leq \kappa_0$.
Applying Theorem \ref{T:expmoment} for $j\in A$ we
obtain:
\[ \lim_{t \to \infty} t^{\kappa_1} \P \left( I_{e_{\lambda_j}}(\alpha
  \xi) > |x|^{- \alpha} t \right) = 0. \]
Hence
\begin{equation}
  \label{e:jsmall}
  \begin{aligned}
    & \lim_{t \to \infty} t^{\kappa_1} \sum_{j: \kappa_j \leq \kappa_0}
    m_j(x / |x|) \left( \int_D m_j \rd \sigma \right)
    \P\left(I_{e_{\lambda_j}}(\alpha \xi) > |x|^{-\alpha} t\right) \\
    = & \frac{1}{\alpha \phi'(\alpha \kappa_1)} \E \left[
      I_{e_{\lambda_1}}(\alpha \xi)^{\kappa_1 - 1} \right] M(x)
    |x|^{\alpha \kappa_1}.
  \end{aligned}
\end{equation}

Now we consider the summation over the $j\in A^{c}$, that is, for $\kappa_j >
\kappa_0$. By the Markov and H\"older inequalities,
\begin{align*}
   t^{\kappa_0} \P \left( I_{e_{\lambda_j}}(\alpha \xi) > |x|^{-
    \alpha} t \right)
  \leq & |x|^{\alpha \kappa_0} \E \left[ \left( \int_0^{e_{\lambda_j}}
         \exp(\alpha \xi_s) \rd s \right)^{\kappa_0} \right] \\
  \leq & |x|^{\alpha \kappa_0} \E \left[ \left( e_{\lambda_j}
         \right)^{\kappa_0 - 1} \int_0^{e_{\lambda_j}} \exp \left(
         \alpha \kappa_0 \xi_s \right) \rd s \right].
\end{align*}
Using the independence of $e_{\lambda_j}$ and $\xi_s$, we have
\begin{align*}
   t^{\kappa_0} \P \left( I_{e_{\lambda_j}}(\alpha \xi) > |x|^{-
    \alpha} t \right)
\leq & |x|^{\alpha \kappa_0} \E \left[ \left( e_{\lambda_j}
         \right)^{\kappa_0} \int_0^{e_{\lambda_j}} \e^{s \phi(\alpha
         \kappa_0)} \rd s \right] \\
  \leq & |x|^{\alpha \kappa_0} \E \left[ \left( e_{\lambda_j}
         \right)^{\kappa_0} \e^{\phi(\alpha \kappa_0) e_{\lambda_j}}
         \right] \\
  = & |x|^{\alpha \kappa_0} \int_0^{\infty} u^{\kappa_0}
      \e^{\phi(\alpha \kappa_0) u} \lambda_j \e^{-\lambda_j u} \rd u
\\
  = & |x|^{\alpha \kappa_0} \frac{\lambda_j \Gamma(\kappa_0 +
      1)}{(\lambda_j - \phi(\alpha \kappa_0))^{\kappa_0 + 1}} \\
  \leq & c |x|^{\alpha \kappa_0} \lambda_j^{- \kappa_0}
\end{align*}
for some constant $c > 0$. By 
\eqref{e:lbev}, \eqref{e:ubef} and the fact $\kappa_0 > 2 \beta$, 
we have
\begin{equation}
  \label{e:jlarge}
  \sum_{j: \kappa_j > \kappa_0}\left| m_j(x / |x|) \right| \cdot
  \left| \int_D m_j \rd
  \sigma \right| \P\left(I_{e_{\lambda_j}}(\alpha \xi) > |x|^{-\alpha}
  t\right)
\leq  c |x|^{\alpha
  \kappa_0} \left( \sum_j j^{- \frac{\kappa_0 - \beta}{\beta}} \right)
t^{-\kappa_0}.
\end{equation}
Combining \eqref{e:jsmall} and \eqref{e:jlarge} completes the proof.
\end{proof}

\end{document}